\documentclass[a4paper,11pt,leqno]{article}
\author{F\'elix del Teso M\'endez}
\usepackage[usenames,dvipsnames,svgnames,table]{xcolor}
\usepackage{amsmath,amsthm,amssymb,verbatim,tabularx,graphicx}
\usepackage{fancyhdr}
\usepackage{enumerate}
\usepackage{tocbibind}
\usepackage{epic}
\usepackage{pgf,tikz}
\usetikzlibrary{arrows}
\newtheorem{teor}{Theorem}[section]

\newtheorem{nota}{Note}

\newtheorem{coro}[teor]{Corollary}

\newcommand\RR{\mathbb{R}^N}

\newcommand\rr{\mathbb{R}}

\newcommand\laph{(-\Delta)^{1/2}}
\newcommand\lap{(-\Delta)^{\sigma/2}}

\newcommand\lra{\longrightarrow}

\numberwithin{equation}{section}

\begin{document}

\begin{center}
{\LARGE\textbf{Finite difference method for a \\[6pt] fractional porous medium equation}}

\vspace{.5cm}

\Large{ by \ F\'elix del Teso \\[6pt]
\em Universidad Aut\'onoma de Madrid}
\vspace{1cm}

\date{}
\end{center}\begin{abstract}
We formulate a numerical method to solve the porous medium type equation with fractional diffusion 
\[\displaystyle\frac{\partial u}{\partial t}+(-\Delta)^{1/2} (u^m)=0.\]
The problem is posed in $x\in\mathbb{R}^N$, $m\geq 1$ and with nonnegative initial data. The fractional Laplacian is implemented via the so-called Caffarelli-Silvestre extension. We prove existence and uniqueness of the solution of this method and also the convergence to the theoretical solution of the equation. We run numerical experiments on typical initial data as well as a section that summarizes and concludes the proposed  method.
\end{abstract}

\noindent {\sc Authors' address: }  Departamento de Matem\'aticas, Universidad Aut\'onoma de Madrid,\\
Campus de Cantoblanco, 28049 Madrid, Spain.

{E-mail address: }\textsf{felix.delteso@uam.es}.

{Telephone number:} +34 91 497 8430.


\section{Introduction}
This paper is concerned with a numerical method for the Cauchy problem
\begin{equation}\label{problem}
\left\{
\begin{array}{ll}
\displaystyle\frac{\partial u}{\partial t}+\laph ( |u|^{m-1}u)=0,& x\in \RR,\ t>0,\\[3mm]
u(x,0)=f(x), & x \in \RR,
\end{array}
\right.
\end{equation}
for exponents $m\geq1$ and space dimension $N\geq1$. We present a numerical method for this equation. We also prove existence and uniqueness of solution to the method, via a maximum principle. Moreover, convergence to the theoretical solution is also proven. We extend later the results for equations with $\varphi(u)$ instead of  $ |u|^{m-1}u$ (in the following $u^m$), for some monotone  $\varphi$ with good regularity conditions. The general theory of existence, uniqueness and regularity of solutions for the equation (\ref{problem}) has been studied by A. de Pablo, F. Qui\'os, A. Rodr\'iguez and J.L. V\'azquez in \cite{afracpor}. They also study a more general case with $\lap$ for $\sigma \in (0,2)$ in \cite{afracpor2}. Even more, in \cite{afracpor3} they study the case where the diffusion is logarithmic, that is, $\varphi(u)=\log(u+1)$ as de natural limit as $m\to 0$ of $\varphi(u)=u^m$. 

We recall that the nonlocal operator $\laph$ is well defined via Fourier transform for any function $f$ in the Schwartz class as the operator such that
\[\mathcal{F}(\laph u)(\xi)=|\xi|\mathcal{F}(u)(\xi)\]
or via Riesz potential, for a more general class of functions, as
\[\laph f(x)=C_N \mbox{P.V.} \int_{\RR}\frac{f(x)-f(y)}{|x-y|^{N+1}}dy\]
where $C_N=\pi^{-\frac{N+1}{2}}\Gamma(\frac{N+1}{2})$ is a normalization constant. For an equivalence of both formulations see for example \cite{val}.

Previous works in numerical analysis for nonlocal equations of this type are done by S. Cifani, E. R. Jakobsen, and Karlsen in \cite{jakob}, \cite{jakob2}, \cite{jakob3}. In particular they formulate some convergent numerical methods for entropy and viscosity solutions. One of the main differences of the present work is that we don not deal directly with the integral formulation of the fractional laplacian, instead of this, we pass through the Caffarelli-Sylvestre extension (\cite{caff}) with implies solving a problem with only local operators in one more space dimension. 

\section{Local formulation of the non-local problem}
\subsection{The problem in  $\RR$}
Our aim is to find numerical approximations for the solutions of the next porous medium equation with fractional diffusion,
\begin{equation}\label{problem}
\left\{
\begin{array}{ll}
\displaystyle\frac{\partial u}{\partial t}(x,t)+\laph u^m(x,t)=0& x\in \RR,\ t>0,\\[3mm]
u(x,0)=f(x) & x \in \RR,
\end{array}
\right.
\end{equation}
with $m\geq 1$ and the initial function $f\in L^1(\RR) \cap L^\infty(\RR)$  and nonnegative. The general theory for existence, uniqueness and regularity of the solution of problem (\ref{problem}) can be found in \cite{afracpor}. In particular, they state that problem (\ref{problem}) is equivalent to the so-called extension formulation,

\begin{equation}\label{problemext}\displaystyle
\left\{ \begin{array}{ll}\displaystyle
\Delta w(x,y,t)=0,& x \in \RR, \ y>0,\ t>0, \\[3mm]
\displaystyle\frac{\partial w^{1/m}}{\partial t}(x,0,t)=\frac{\partial w}{\partial y}(x,0,t),& x \in \RR,\ y=0,\ t>0, \\[3mm]
w(x,0,0)=f^m(x) , & x\in \RR.
\end{array} \right.
\end{equation}
The equivalence between \eqref{problem} and \eqref{problemext} holds in the sense of trace and harmonic extension operators, that is,
\[ u(x,t)=Tr(w^{1/m}(x,y,t)), \ \ \  w(x,y,t)=E(u^m(x,t)).\]
In (\ref{problemext}), $\Delta$ denotes the $N+1$ dimensional laplacian operator,
\[\Delta=\sum_{l=1}^N\frac{\partial^2}{\partial x_l^2}+\frac{\partial^2}{\partial y^2}.\]

\subsection{The problem in the bounded domain}
In order to construct a numerical solution to Problem (\ref{problemext}), we perform a monotone approximation of the solutions  in the whole space by the solutions of the problem posed a bounded domain.  

We consider $X_l,Y,T\in \rr$ positive for $l=1, . . . , N$. We define the bounded domain $\Omega=(-X_1,X_1)\times . . . \times (-X_N,X_N)\times(0,Y)$, and $\Gamma=\partial \Omega$. For convenience we also divide the boundary in two parts: 
\[\Gamma_d=[-X_1,X_1]\times . . . \times [-X_N,X_N]\times \{0\},\] and $\Gamma_h=\partial \Omega \backslash \Gamma_d $. With these notations, we formulate the problem in the bounded domain as

\begin{equation}\label{problemextbdd}\displaystyle
\left\{ \begin{array}{ll}\displaystyle
\Delta w(x,y,t)=0,& (x,y) \in \Omega ,\ t\in(0,T], \\[3mm]
\displaystyle\frac{\partial w^{1/m}}{\partial t}(x,0,t)=\frac{\partial w}{\partial y}(x,0,t),& (x,y) \in \Gamma_d,\ t\in(0,T], \\[3mm]
w(x,0,0)=f^m(x) , & (x,y)\in \Gamma_d,\\[3mm]
w(x,y,t)=0, & (x,y)\in \Gamma_h,
\end{array} \right.
\end{equation}
where we have imposed homogeneous boundary conditions on $\Gamma_h$. 

In the sequel, we will consider the problem with $N=1$ in order to simplify de notation but all the arguments are also valid for $N>1$ without any extra effort.

\begin{picture}(350,200)

\put(0, 15){\line(1, 0){350 }}

\put(175, 0){\line(0, 1){180 }}

\put(20,15){\dashbox{5}(310,140){}}

\put(20, 15.5){\line(1, 0){310 }}
\put(20, 14.5){\line(1, 0){310 }}

\put(20,155.5){\dashbox{5}(310,0){}}
\put(20,154.5){\dashbox{5}(310,0){}}

\put(19.5,15){\dashbox{5}(0,140){}}
\put(20.5,15){\dashbox{5}(0,140){}}

\put(330.5,15){\dashbox{5}(0,140){}}
\put(329.5,15){\dashbox{5}(0,140){}}

\put(325,0){\shortstack[l]{X}}

\put(10,0){\shortstack[l]{-X}}

\put(125, 80){\shortstack[l]{\huge$\Omega$}}

\put(180,160){\shortstack[l]{Y}}

\put(250, 180){\line(1, 0){40 }}
\put(250, 180.5){\line(1, 0){40 }}
\put(250, 179.5){\line(1, 0){40 }}
\put(295,178){$\Gamma_d$}

\put(250, 170){\dashbox{5}(40,0){}}
\put(250, 169.5){\dashbox{5}(40,0){}}
\put(250, 170.5){\dashbox{5}(40,0){}}
\put(295,166){$\Gamma_h$}

\end{picture}

\section{Discrete formulation}

In order to solve problem (\ref{problemextbdd}) fot $t\in[0,T]$, we first perform a space and time discretization.  For time discretization we choose the number of steps $J$, and then
\[0\leq j \Delta t\leq T, \ \ \ j=0, . . .   ,J \ \ \mbox{where } \Delta t=T/J \mbox{ and } t_j=j\Delta t.\]

We also need to discretize the space domain  $\overline{\Omega}=[-X,X]\times[0,Y]$. Let $I,K$ be the number  of steps on each space direction, 
\[0 \leq i \Delta x  \leq 2X , \ \ \ i=0, . . .   ,I \ \ \mbox{where } \Delta x=2X/I \mbox{ and } x_i= i \Delta x-X, \]
\[0 \leq k \Delta y \leq Y , \ \ \ k=0, . . .   ,K \ \ \mbox{where } \Delta y=Y/K \mbox{ and } y_k= k \Delta y,\]

\begin{picture}(350,190)

\put(15, 15){\line(1, 0){320 }}

\put(175, 0){\line(0, 1){180 }}

\put(20,15){\dashbox{5}(310,140){}}

\put(20,15.5){\dashbox{5}(310,0){}}
\put(20,14.5){\dashbox{5}(310,0){}}

\put(20,155.5){\dashbox{5}(310,0){}}
\put(20,154.5){\dashbox{5}(310,0){}}

\put(19.5,15){\dashbox{5}(0,140){}}
\put(20.5,15){\dashbox{5}(0,140){}}

\put(330.5,15){\dashbox{5}(0,140){}}
\put(329.5,15){\dashbox{5}(0,140){}}

\put(125, 80){\shortstack[l]{\huge$\Omega$}}


\put(20,15){\grid(310,140)(15.5,15.5)}
\put(14,5){\shortstack[l]{$x_0$}}
\put(30,5){\shortstack[l]{$x_1$}}
\put(46,5){\shortstack[l]{$x_2$}}
\put(62,5){\shortstack[l]{$. . . $}}
\put(77,5){\shortstack[l]{$x_i$}}
\put(91,5){\shortstack[l]{$. . . $}}

\put(325,5){\shortstack[l]{$x_I$}}

\put(5, 15){\shortstack[l]{$y_0$}}
\put(5, 29){\shortstack[l]{$y_1$}}
\put(5, 43){\shortstack[l]{$y_2$}}
\put(8, 55){\shortstack[l]{$.$}}
\put(8, 59){\shortstack[l]{$.$}}
\put(8, 63){\shortstack[l]{$.$}}
\put(5, 70){\shortstack[l]{$y_k$}}
\put(8, 84){\shortstack[l]{$.$}}
\put(8, 88){\shortstack[l]{$.$}}
\put(8, 92){\shortstack[l]{$.$}}

\put(5, 155){\shortstack[l]{$y_K$}}
\end{picture}
\\\\
We use the notation 
\begin{equation}\label{notation}
w(x_i,y_k,t_j)=(w_j)_i^k,
\end{equation}
for the value of the solution $w$ to Problem (\ref{problemextbdd}) in the points of the mesh, and
\begin{equation}\label{notation}
w(x_i,y_k,t_j)\approx(W_j)_i^k,
\end{equation}
for the solution of the numerical method.

\subsection{Numerical Method}
In the following, we will assume that $\Delta y=\Delta x$. For each time step $j=1,. . . ,J$ we have to solve the following linear system of equations
\begin{equation}\label{numericmethod}\displaystyle
\left\{ \begin{array}{ll}\displaystyle
\frac{(W_j)_{i+1}^{k}+(W_j)_{i-1}^{k}+(W_j)_{i}^{k+1}+(W_j)_{i}^{k-1}-4(W_j)_{i}^{k}}{\Delta x^2} =0, &  0<i<I,
 0<k<K,\\[3mm]\displaystyle
(W_j)_i^0=\bigg[\frac{\Delta t}{\Delta x} \big({(W_{j-1})}_i^1-{(W_{j-1})}_i^0\big)  +[(W_{j-1})_i^0]^{1/m}\bigg]^m, & \mbox{if } 0<i<I,\\[3mm]\displaystyle
(W_j)_i^k=0,& \mbox{otherwise. }\\
\end{array} \right.
\end{equation}
Note that the second equation  is explicit in the sense that it only depends on the solution where the solution of the numerical method in the previous time step. We use the solution of 
\begin{equation}\nonumber\displaystyle
\left\{ \begin{array}{ll}\displaystyle
\frac{(W_0)_{i+1}^{k}+(W_0)_{i-1}^{k}+(W_0)_{i}^{k+1}+(W_0)_{i}^{k-1}-4(W_0)_{i}^{k}}{\Delta x^2} =0, &  0<i<I,
 0<k<K,\\[3mm]\displaystyle
(W_0)_i^0=f^m(x_i), & \mbox{if } 0<i<I,\\[3mm]\displaystyle
(W_0)_i^k=0,& \mbox{otherwise, }\\
\end{array} \right.
\end{equation}
 to start the numerical method.

\subsection{Local truncation error}
We define the local truncation error $(\tau_j)_i^k$ as the error that comes from plugging the solution $w$ to Problem (\ref{problemextbdd}) into the numerical method (\ref{numericmethod}).
Let us also write
\begin{equation}
\Lambda=\max_{i,k,j}|(\tau_j)_i^k|.
\end{equation}
\begin{teor}\label{localtruncerror}
Let $w$ be the solution to Problem (\ref{problemextbdd}). Assume that there exist two constants $C_1,C_2>0$ such that 
\[C_1\Delta x\leq \Delta t\leq C_2\Delta x. \]
Then,
\begin{equation}
\Lambda=O(\Delta t(\Delta x+\Delta t)).
\end{equation}
\end{teor}
\begin{proof}
Of course the local truncation error in the boundary nodes situated on the part $\Gamma_h$ of the boundary  is zero since we have imposed that the solution is zero in $\Gamma_h$ and equal to $f^m(x)$ in $\Gamma_d$ as in Problem (\ref{problemextbdd}).

If $0<i<I$ and $0<k<K$ (the interior nodes), then
\begin{eqnarray*}
(\tau_{j-1})_i^k&=&\frac{1}{\Delta x^2} \big[(w_{j-1})_{i+1}^{k}+(w_{j-1})_{i-1}^{k}+(w_{j-1})_{i}^{k+1}+(w_{j-1})_{i}^{k-1}-4(w_{j-1})_{i}^{k}\big]\\
&=& \Delta w(x_i,y_k,t_{j-1})+O(\Delta x^2)=O(\Delta t(\Delta x+ \Delta t)).
\end{eqnarray*}

If $0< i< I$ and $k=0$ (the $\Gamma_h$ nodes), then
\begin{eqnarray*}
(\tau_{j-1})_i^0&=&\frac{\Delta t}{\Delta x} \big[{(w_{j-1})}_i^1-{(w_{j-1})}_i^0\big]  +[(w_{j-1})_i^0]^{1/m}-[(w_{j})_i^0]^{1/m}\\
&=&\Delta t\big[ \frac{\partial w}{\partial y}(x_i,0,t_{j-1})+O(\Delta x)\big]- \Delta t \big[\frac{\partial w^{1/m}}{\partial t}(x_i,0,t_{j-1})+O(\Delta t) \big]\\
&=&O(\Delta t \Delta x)+ O(\Delta t^2)=O(\Delta t(\Delta t+\Delta x)).
\end{eqnarray*}
The previous calculus are done assuming that the theoretical solution is smooth enough.
\end{proof}
\subsection{Existence and uniqueness of the numerical solution}

We define the following quantity needed for the maximum principle theorem,

\[ b_{max}= \max_{x}\{f^m(x)\}.\] 
In the sequel we will sometimes denote by $\varphi(x)=x^m$. If $m\geq1$ then $\varphi'(x)=mx^{m-1}$ is a locally bounded function for $x\geq0$.
\begin{teor}\label{maxprin}[Discrete maximum principle]
Let $(W_j)_i^k$ be the solution to  Problem (\ref{numericmethod}) with $m\geq1$. Assume 
\begin{equation}\label{constant}
\Delta t\leq C(m,f) \Delta x, \qquad \mbox{where} \qquad C(m,f)=[m (b_{max})^{(m-1)}]^{-1}.
\end{equation}
Then, for every $i,k,j$ we have
\begin{equation}
0\leq (W_j)_i^k \leq b_{max}.
\end{equation}
\end{teor}
\begin{nota}
It is interesting to remark that $C(1,f) =1$, which means that we recover the expected restriction $\Delta t\leq \Delta x$ for the linear case. In the literature, this kind of condition use to be called \textbf{CFL} condition.
\end{nota}
\begin{proof}
On each time step we have a discrete harmonic extension problem and so, it is sufficient to prove the maximum principle in the boundary nodes and therefore the interior nodes are automatically smaller than them.
We will do the proof by induction on each time step. It is trivial that
\[0\leq (W_0)_i^k\leq b_{max}.\]
Now we assume that
\[0\leq (W_{j-1})_i^k\leq b_{max}.\]
Then, 
\[[(W_j)_i^0]^{1/m}=\frac{\Delta t}{\Delta x} \big({(W_{j-1})}_i^1-{(W_{j-1})}_i^0\big)  +[(W_{j-1})_i^0]^{1/m},\]
and changing variables to $\displaystyle(U_j)_i^k= [(W_j)_i^k]^{1/m} $, we obtain that
\begin{equation}\label{recu}
(U_j)_i^0=\frac{\Delta t}{\Delta x} \big([(U_{j-1})_i^1]^m-[(U_{j-1})_i^0]^m\big)  +(U_{j-1})_i^0.
\end{equation}
Now, using the Mean Value Theorem
\[[(U_{j-1})_i^1]^m-[(U_{j-1})_i^0]^m=[(U_{j-1})_i^1-(U_{j-1})_i^0]\varphi'(\xi),\]
for some $\xi \in [(U_{j-1})_i^1,(U_{j-1})_i^0]$. Then we can rewrite (\ref{recu}) as
\begin{equation}\label{recu2}
(U_j)_i^0=\varphi'(\xi)\frac{\Delta t}{\Delta x} (U_{j-1})_i^1 +\bigg[1-\varphi'(\xi)\frac{\Delta t}{\Delta x}\bigg](U_{j-1})_i^0.
\end{equation}
At this point, thanks to our induction hypothesis and the value of the constant (\ref{constant}), it follows that $\displaystyle \varphi'(\xi)\frac{\Delta t}{\Delta x}=m\xi^{m-1}\frac{\Delta t}{\Delta x}\leq 1$ and therefore
\begin{eqnarray*}
|(U_j)_i^0|&\leq&\varphi'(\xi)\frac{\Delta t}{\Delta x} |(U_{j-1})_i^1| +\bigg[1-\varphi'(\xi)\frac{\Delta t}{\Delta x}\bigg]|(U_{j-1})_i^0|\\
&\leq&\varphi'(\xi)\frac{\Delta t}{\Delta x} (b_{max})^{1/m} +\bigg[1-\varphi'(\xi)\frac{\Delta t}{\Delta x}\bigg] (b_{max})^{1/m} \\
&=&(b_{max})^{1/m}.
\end{eqnarray*}
The same argument holds for $ (U_j)_i^0\geq 0$.
\end{proof}

\begin{coro}
If $\Delta t\leq C(m,f) \Delta x$, then Problem (\ref{numericmethod}) has a unique solution.
\end{coro}
\begin{proof}
We are only going to proof uniqueness, and since we are working with a linear system of equations, the existence is equivalent to the uniqueness.

Let $(W_j)_i^k$ and $(V_j)_i^k$ be two solutions of (\ref{numericmethod}) and let us define 
\[(Y_j)_i^k=(W_j)_i^k-(V_j)_i^k.\]
Then $(Y_0)_i^k$ satisfies  (\ref{numericmethod})  with $f\equiv 0$ and so, by the discrete maximum principle,
\[0\leq (Y_0)_i^k\leq0,\]
for all $i,k$. Proceeding by induction we get that for all $i,j,k$ we have
\[ (Y_j)_i^k=0. \]
\end{proof}

\subsection{Convergence of the numerical solution}
Since we are originally interested in the solution at the boundary $(U_j)_i^k=[(W_j)_i^k]^{1/m}$ we have two options in order to define the define de error of the numerical method. The first option,
\begin{equation}\label{error2}
(f_j)_i^k=w(x_i,y_k,t_j)-(W_j)_i^k, \ \ \ F_j=\max_{i,k}|(f_j)_i^k|,
\end{equation}
and the second one
\begin{equation}\label{error1}
(e_j)_i^k=u(x_i,y_k,t_j)-(U_j)_i^k, \ \ \ E_j=\max_{i,k}|(e_j)_i^k|.
\end{equation}
Anyway, if we are able to control (\ref{error1}) we have also a control of (\ref{error2}) because $\varphi'(x)$ is locally bounded and so
\begin{eqnarray*}
(f_j)_i^k&=&(w_j)_i^k-(W_j)_i^k=[(u_j)_i^j]^m-[(U_j)_i^k]^m\\
&=&\big[(u_j)_i^j-(U_j)_i^k\big]\varphi'(\xi)\\
&=& (e_j)_i^k\varphi'(\xi),
\end{eqnarray*}
for some $\xi \in [(u_j)_i^j,(U_j)_i^k]$. This implies that $|(f_j)_i^k|\leq C(m,f)|(e_j)_i^k|$ and therefore $F_j\leq C(m,f) E_j$.
\begin{teor}\label{convergence}
Let $w$ be the solution to Problem (\ref{problemextbdd}) and $(W_j)_i^k$ be the solution to system (\ref{numericmethod}) with $m\geq 1$. Assume that there exists two constants $C(m,f), D>0$ such that 
\[D\Delta x\leq \Delta t\leq C(m,f)\Delta x. \]
Then
\[F_j=O(\Delta x+\Delta t),\]
for $j=1,. . .,J$.
\end{teor}
\begin{proof}
As in the local truncation error, the election for the boundary conditions in $\Gamma_h$ in the numerical method give us error zero there. 

Lets us denote $(E_B)_j$ and $(E_I)_j$ the maximum errors in the boundary nodes and in the interior nodes at time $j\Delta t$, that is

\[(E_B)_j=\max_{0\leq i\leq I}|(e_j)_i^0|.\]
Since we have chosen a second order approximation for the laplacian, it is well known that

\[E_j=\max\{(E_B)_j, O(\Delta x^2)\}\leq (E_B)_j+O(\Delta x^2).\]

If $0\leq i\leq I$, we have the equations
\[(\tau_{j-1})_i^0=\frac{\Delta t}{\Delta x} \big[{(w_{j-1})}_i^1-{(w_{j-1})}_i^0\big]  +[(w_{j-1})_i^0]^{1/m}-[(w_{j})_i^0]^{1/m},\]

\[[(W_j)_i^0]^{1/m}=\frac{\Delta t}{\Delta x} \big[{(W_{j-1})}_i^1-{(W_{j-1})}_i^0\big]  +[(W_{j-1})_i^0]^{1/m}.\]
Rewriting the above equations en terms of $u$ and $(U_j)_i^k$ we get
\[(\tau_{j-1})_i^0=\frac{\Delta t}{\Delta x} \big[[{(u_{j-1})}_i^1]^{m}-[{(u_{j-1})}_i^0]^{m}\big]  +(u_{j-1})_i^0-(u_{j})_i^0,\]

\[(U_j)_i^0=\frac{\Delta t}{\Delta x} \big[[{(U_{j-1})}_i^1]^{m}-[{(U_{j-1})}_i^0]^{m}\big]  +(U_{j-1})_i^0.\]
Subtracting them, we obtain
\begin{equation}\label{errorequation}
(e_j)_i^0=\frac{\Delta t}{\Delta x} \bigg[\big([{(u_{j-1})}_i^1]^{m}-[{(U_{j-1})}_i^1]^{m}\big)-\big([{(u_{j-1})}_i^0]^{m}-[{(U_{j-1})}_i^0]^{m}\big)\bigg]  +(e_{j-1})_i^0-(\tau_{j-1})_i^0.
\end{equation}
By mean value theorem, relation (\ref{errorequation}) turns into
\[(e_j)_i^0=\frac{\Delta t}{\Delta x} \bigg[\big({(u_{j-1})}_i^1-{(U_{j-1})}_i^1\big)\varphi'(\xi_1)-\big({(u_{j-1})}_i^0-{(U_{j-1})}_i^0\big)\varphi'(\xi_0)\bigg]  +(e_{j-1})_i^0-(\tau_{j-1})_i^0,\]
for some $\xi_0\in [{(u_{j-1})}_i^0,{(U_{j-1})}_i^0\big]$ and $\xi_1\in [{(u_{j-1})}_i^1,{(U_{j-1})}_i^1\big]$. Then

\[(e_j)_i^0=\frac{\Delta t}{\Delta x} \bigg[{(e_{j-1})}_i^1\varphi'(\xi_1)-{(e_{j-1})}_i^0\varphi'(\xi_0)\bigg]  +(e_{j-1})_i^0-(\tau_{j-1})_i^0,\]
that is,
\[(e_j)_i^0=\frac{\Delta t}{\Delta x}\varphi'(\xi_1){(e_{j-1})}_i^1  +\bigg[1-\frac{\Delta t}{\Delta x}\varphi'(\xi_0)\bigg](e_{j-1})_i^0-(\tau_{j-1})_i^0.\]
By assumption, all the coefficients that comes with $(e_{j-1})_i^k$ are positive, then

\begin{eqnarray}\label{eqerror2}
|(e_j)_i^0|&\leq&\frac{\Delta t}{\Delta x}\varphi'(\xi_1)|{(e_{j-1})}_i^1|  +\bigg[1-\frac{\Delta t}{\Delta x}\varphi'(\xi_0)\bigg]|(e_{j-1})_i^0|+\Lambda\nonumber\\
&\leq&\frac{\Delta t}{\Delta x}\varphi'(\xi_1)E_{j-1}  +\bigg[1-\frac{\Delta t}{\Delta x}\varphi'(\xi_0)\bigg]E_{j-1}+ \Lambda.
\end{eqnarray}
So the key point is to be able to control the difference between $\varphi'(\xi_1)$ and $\varphi'(\xi_0)$ and this is not difficult at all.
Since there exists a constant (assuming enough regularity of the solution $u$) $K\geq0$ such that
\[|(u_j)_i^1-(u_j)_i^0|\leq K\Delta x\ \  \mbox{ and }\ \ |(U_j)_i^1-(U_j)_i^0|\leq K\Delta x.\]
then (This formula is valid only for $m$ natural, and in this case $\varphi'(\xi_1)=\sum_{l=1}^m [(u_{j})_i^1]^{m-l}[(U_{j})_i^1]^{l-1}$ and $\varphi'(\xi_0)=\sum_{l=1}^m [(u_{j})_i^0]^{m-l}[(U_{j})_i^0]^{l-1}$)
\begin{eqnarray*}
(K_{j-1})_i^1&=&\sum_{l=1}^m [(u_{j})_i^1]^{m-l}[(U_{j})_i^1]^{l-1}\\
&\leq&\sum_{l=1}^m [(u_{j})_i^0+K\Delta x]^{m-l}[(U_{j})_i^0+K\Delta x]^{l-1}\\
&\leq&\sum_{l=1}^m [(u_{j})_i^0]^{m-l}[(U_{j})_i^0]^{l-1}+L\Delta x\\
&=&(K_{j-1})_i^0+R\Delta x.
\end{eqnarray*}
But it is not difficult to prove that in fact $|\varphi'(\xi_1)- \varphi'(\xi_0)|\leq R\Delta x$
where $R\geq0$ is a constant depending only on $m$, $K$, and $b_{max}$. The proof is only based in the idea that,  the function
\[g(x,y)=\frac{x^m-y^m}{x-y},\] 
is $C^1((0,\infty)\times (0\times \infty))$, and $\varphi'(\xi_0)=g((u_{j})_i^0,(U_{j})_i^0)$ and $\varphi'(\xi_1)=g((u_{j})_i^1,(U_{j})_i^1)$.
 Then from (\ref{eqerror2}) we obtain
\begin{eqnarray*}
|(e_j)_i^0| &\leq&\frac{\Delta t}{\Delta x}\bigg[\varphi'(\xi_0)+ R \Delta x \bigg]E_{j-1}  +\bigg[1-\frac{\Delta t}{\Delta x}\varphi'(\xi_0)\bigg]E_{j-1}+\Lambda\\
&=&\frac{\Delta t}{\Delta x}R\Delta x E_{j-1} +  E_{j-1}+\Lambda\\
&\leq& (1+R\Delta t)E_{j-1}+\Lambda.
\end{eqnarray*}
Remember also that we have $\displaystyle E_j\leq \max_{0<i<I} |(e_j)_i^0|+O(\Delta x^2)$ and $\Lambda =O(\Delta x^2)$. Then we have the next recurrence equation for the error
\begin{equation}\label{errorrec1}
E_j\leq (1+R\Delta t)E_{j-1}+\Lambda.
\end{equation}
We should also remember that $\Delta t=T/J$ where $T$ was the final time and $J$ the number of elements in the time discretization. Then, we can rewrite (\ref{errorrec1}) as
\begin{equation}\label{errorrec2}
E_j\leq (1+C\frac{1}{J})E_{j-1}+\Lambda,
\end{equation}
for some constant $C>0$.

Of course it is enough to bound $E_{J}$ to ensure the convergence of the method. 
Since $\Lambda = O(\Delta t^2)$ lets say that there exists a constant $L>0$ such that $\Lambda \leq L \Delta t^2$. Then, from (\ref{eqerror2}) we obtain
\begin{eqnarray*}
E_J&\leq&(1+C\frac{1}{J})E_{J-1}+L \Delta t^2\\
&\leq&(1+C\frac{1}{J})\bigg[E_{J-1}+L \Delta t^2\bigg]\\
&\leq&(1+C\frac{1}{J})\bigg[(1+C\frac{1}{J})E_{J-2}+L \Delta t^2+L \Delta t^2\bigg]\\
&=&(1+C\frac{1}{J})^2\bigg[E_{J-2}+2L \Delta t^2\bigg]\\
&\leq& . . . \leq (1+C\frac{1}{J})^J\bigg[E_0+L J \Delta t^2\bigg].
\end{eqnarray*}
But $(1+C\frac{1}{J})^J\leq e^{C}$, $J\Delta t=T$ and $E_0\leq D\Delta t^2 $ for some $D>$, so
\[E_J\leq e^C\bigg[D\Delta t^2 +L T \Delta t\bigg],\]
that is
\[E_J=O(\Delta t)=O(\Delta x + \Delta t).\]
\end{proof}
\subsection{Properties of the scheme}
We present now some properties that can be deduced from the numerical scheme (\ref{numericmethod}). They are the analougus of some of the energy stimulates presented in \cite{afracpor} and \cite{afracpor2} for the fractional porous medium equation. The first one is a  consequence of the  discrete maximum principle.

\begin{coro}{[Comparison principle]}\label{comprin}
Let $f,g \in C^2([-X,X])$ nonnegative such that $f(x)\geq g(x)$ $\forall x\in [-X,X]$ and let also $(W_j)_i^k$ and  $(Z_j)_i^k$ be the correspondent solutions of Problem (\ref{numericmethod}) with $m\geq1$. If $\Delta t\leq C(m,f) \Delta x$ then, 
\[(W_j)_i^k \geq(Z_j)_i^k \ \ \ \forall i,j,k.\]
\end{coro}
\begin{proof}
Let $(H_j)_i^k$ be the solution of Problem (\ref{numericmethod}) with a nonnegative initial data given by $h=(f^m-g^m)^{1/m}$.  By the Discrete  Maximum Principle \ref{maxprin}, we have that $(H_j)_i^k\geq 0$. We recall that at time $j=0$ the scheme is linear and so $(W_0)_i^k-(Z_0)_i^k=(H_0)_i^k\geq 0$. Proceeding by induction, assume that $(W_j)_i^k-(Z_j)_i^k\geq0$ or equivalently $(U_j)_i^k-(V_j)_i^k\geq0$. The next two equations holds
\[(U_{j+1})_i^0=\frac{\Delta t}{\Delta x} \big([{(U_{j})}_i^1]^{m}-[{(U_{j})}_i^0]^{m}\big)  +(U_{j})_i^0.\]
\[(V_{j+1})_i^0=\frac{\Delta t}{\Delta x} \big([{(V_{j})}_i^1]^{m}-[{(V_{j})}_i^0]^{m}\big) +(U_{j})_i^0.\]
subtracting them and using the mean value theorem,
\begin{eqnarray*}
(U_{j+1})_i^0-(V_{j+1})_i^0&=&\frac{\Delta t}{\Delta x} \big(\varphi'(\xi_1)\left({(U_{j})}_i^1-{(V_{j})}_i^1\right) -\varphi'(\xi_0)\left({(U_{j})}_i^0-{(V_{j})}_i^0\right)\big)  +(U_{j})_i^0-(V_{j})_i^0\\
&=&\frac{\Delta t}{\Delta x}\varphi'(\xi_1)\left({(U_{j})}_i^1-{(V_{j})}_i^1\right) +\left[ 1- \frac{\Delta t}{\Delta x}\varphi'(\xi_0)\right] \left((U_{j})_i^0-(V_{j})_i^0\right).
\end{eqnarray*}
And thanks to our CFL condition and the induction hypothesis, al terms in the right hand side of the equations are positive and so $(U_{j+1})_i^0-(V_{j+1})_i^0\geq 0$.

\end{proof}

The following Discrete-$L^1$-Contraction property is the analogous of the one presented in \cite{afracpor} (Theorem 6.2). 

\begin{teor}{[$L^1$-Contraction]}
Under the assumptions of Corollary \ref{comprin} let $(U_j)_i^k=[(W_j)_i^k]^{1/m}$ and $(V_j)_i^k=[(Z_j)_i^k]^{1/m}$. The next contractions property holds,
\begin{equation}\label{l1cont}
\sum_{i=i}^{I-1} \left[(U_j)_i^0-(V_j)_i^0 \right]\leq\sum_{i=i}^{I-1} \left[(U_{j-1})_i^0-(V_{j-1})_i^0 \right],
\end{equation}
for all $j=1,. . . ,J$.
\end{teor}

\begin{nota}
A mass decay property is a direct consequence of (\ref{l1cont}).
\end{nota}

\begin{proof}
We will prove the result for $j=1$ for simplicity, but and induction method could be used to prove it for a general $j$. We have the next two relations for the solutions at the boundary

\[
(U_1)_i^0=\frac{\Delta t}{\Delta x} \big((W_{0})_i^1-f^m(x_i)\big)  +f(x_i).
\]
\[
(V_1)_i^0=\frac{\Delta t}{\Delta x} \big((Z_{0})_i^1-g^m(x_i)\big)  +g(x_i).
\]
Subtracting both and summing for all $i$, we get

\[\sum_{i=1}^{I-1}[(U_1)_i^1-(V_1)_i^1]=\frac{\Delta t}{\Delta x} \left(\sum_{i=1}^{I-1}[(W_0)_i^0-(Z_0)_i^0]-
\sum_{i=1}^{I-1}[f^m(x_i)-g^m(x_i)] \right)+\sum_{i=1}^{I-1}[f(x_i)-g(x_i)].\]
So, if we prove that 
\[\sum_{i=1}^{I-1}[(W_1)_i^0-(Z_1)_i^0]-\sum_{i=1}^{I-1}[f^m(x_i)-g^m(x_i)] \leq 0,\]
we are done. To prove this, we first recall that, for any $k=1, . . . ,K-1$ we have
\[(W_j)_{i+1}^{k}+(W_j)_{i-1}^{k}+(W_j)_{i}^{k+1}+(W_j)_{i}^{k-1}=4(W_j)_{i}^{k}.\]
\[(Z_j)_{i+1}^{k}+(Z_j)_{i-1}^{k}+(Z_j)_{i}^{k+1}+(Z_j)_{i}^{k-1}=4(Z_j)_{i}^{k}.\]
We call $(H_j)_i^k=(W_j)_{i}^{k}-(Z_j)_{i}^{k}$. Note that, thanks to the comparison principle, $(H_j)_i^k\geq0$. Summing up,
\begin{eqnarray*}
\sum_{i=1}^{I-1}(H_j)_{i}^{k-1}&=&4\sum_{i=1}^{I-1}(H_j)_{i}^{k}-\sum_{i=1}^{I-1}(H_j)_{i+1}^{k}-\sum_{i=1}^{I-1}(H_j)_{i-1}^{k}-\sum_{i=1}^{I-1}(W_j)_{i}^{k+1}\\
&=& 4\sum_{i=1}^{I-1}(H_j)_{i}^{k}-\sum_{i=2}^{I-1}(H_j)_{i}^{k}-\sum_{i=1}^{I-2}(H_j)_{i}^{k}-\sum_{i=1}^{I-1}(H_j)_{i}^{k+1}\\
&=& 2\sum_{i=1}^{I-1}(H_j)_{i}^{k}+(H_j)_{1}^{k}+(H_j)_{I-1}^{k}-\sum_{i=1}^{I-1}(H_j)_{i}^{k+1}.
\end{eqnarray*}
Now, remember that $(W_j)_{i}^{K}, (Z_j)_{i}^{k},(H_j)_{i}^{k}\equiv0$ thanks to our homogenous dirichlet boundary condition. Then, using the above relation with $k=K-1$ we get
\begin{eqnarray}
\sum_{i=1}^{I-1}(H_j)_{i}^{K-2}&=&2\sum_{i=1}^{I-1}(H_j)_{i}^{K-1}+(H_j)_{1}^{k}+(H_j)_{I-1}^{K-1}-\sum_{i=1}^{I-1}(H_j)_{i}^{K}\\
&\geq& \sum_{i=1}^{I-1}(H_j)_{i}^{K-1}\nonumber.
\end{eqnarray}
By induction we get then,
\[\displaystyle \sum_{i=1}^{I-1}[(W_j)_{i}^{k-1} -(Z_j)_{i}^{k-1}]\geq  \sum_{i=1}^{I-1}[(W_j)_{i}^{k}-(Z_j)_{i}^{k}],\]
for all $k=0,. . . ,K-1$, which in particular states that,
\[\displaystyle \sum_{i=1}^{I-1}[f^m(x_i)-g^m(x_i)]\geq  \sum_{i=1}^{I-1}[(W_j)_{i}^{1}-(Z_j)_{i}^{1}].\]
\end{proof}

As we said, this $L^1$ contraction property directly implies a total mass decay property,
\begin{equation}
\sum_{i=i}^{I-1} (U_j)_i^0\leq\sum_{i=i}^{I-1} (U_{j-1})_i^0
\end{equation}
This  mass decay is a consequence of the homogenous dirichlet boundary data that we have artificially imposed. We show that, if we pose the scheme in $\RR$, then the expected conservation of mass holds.

\begin{teor}{[Conservation of mass]}
Let $f\in L^1(\RR)$ nonnegative. Let also $(W_j)_i^k$ be the correspondent solution of Problem (\ref{numericmethod}) posed in $\rr^2_+=(-\infty.\infty)\times[0,\infty)$ with $m\geq1$. Then

\begin{equation}
\sum_{i=i}^{I-1} (U_j)_i^0=\sum_{i=i}^{I-1} (U_{j-1})_i^0
\end{equation}
for all $j=1,. . . ,J$.
\end{teor}

\begin{proof}
Directly from the numerical method we have the following relation for the solution at the boundary for every $i\in \mathbb{Z}$,

\[(U_1)_i^0=\frac{\Delta t}{\Delta x} \big[{(W_{0})}_i^1-{(W_{0})}_i^0\big]  +(U_{0})_i^0.\]
Summing up on all $i\in \mathbb{Z}$,

\[\sum_{i=-\infty}^\infty (U_1)_i^0=\frac{\Delta t}{\Delta x} \left[\sum_{i=-\infty}^\infty{(W_{0})}_i^1-\sum_{i=-\infty}^\infty {(W_{0})}_i^0\right]  +\sum_{i=-\infty}^\infty (U_{0})_i^0\]
so it is enough to prove that  $\displaystyle \sum_{i=-\infty}^\infty{(W_{0})}_i^1=\sum_{i=-\infty}^\infty {(W_{0})}_i^0$ . We know from our numerical scheme that  for any $\forall k\geq 1$ we have
\[(W_0)_{i}^{k}+(W_0)_{i-1}^{k}+(W_0)_{i}^{k+1}+(W_0)_{i}^{k-1}=4(W_0)_{i}^{k},\]
Summing up,
\begin{eqnarray}\label{consmass}
\sum_{i=-\infty}^{\infty}(W_0)_{i}^{k-1}&=&4\sum_{i=-\infty}^{\infty}(W_0)_{i}^{k}-\sum_{i=-\infty}^{\infty}(W_0)_{i+1}^{k}-\sum_{i=-\infty}^{\infty}(W_0)_{i-1}^{k}-\sum_{i=-\infty}^{\infty}(W_0)_{i}^{k+1}\nonumber\\
&=&2\sum_{i=-\infty}^{\infty}(W_0)_{i}^{k}-\sum_{i=-\infty}^{\infty}(W_0)_{i}^{k+1}.
\end{eqnarray}
Now, lets call $\displaystyle a_k=\sum_{i=-\infty}^{\infty}(W_0)_{i}^{k}$. Since we have finite total mass, we can assume that $a_0=1$. Then relation (\ref{consmass}) can be interpreted as the next recurrence succession,

\[ a_{k+2}=2a_{k+1}-a_k \mbox{ with } a_0=1 \]

The general solution of this recurrence is $a_k=c+kd$, where $c$ and $d$ are two constants which depends on the initial condition. But here we have only the initial condition $a_0=1$, which gives
$a_k=1+kd$.
What we want at this point is a second initial conditions $a_1=1$ and then the general solution will be \[\displaystyle\sum_{i=-\infty}^{\infty}(W_0)_{i}^{k}=a_k=1\] 
Proceeding by contradiction, assume $a_1<1$, then 
\[a_k=1+k(a_1-1)\stackrel{k\to\infty}{\lra}-\infty\]
which is a contraction with the maximum principle. The other options is assuming $a_1>1$ but then, you get $a_k\stackrel{k\to\infty}{\lra}\infty$ which is again a contradiction with the result of loosing mass in the bounded domain. So $a_1=1$ is the only option. Then we get the desired result 
\[\sum_{i=-\infty}^\infty{(W_{0})}_i^1=a_1=a_0=\sum_{i=-\infty}^\infty {(W_{0})}_i^0.\]
\end{proof}
\section{Comments}
\subsection{Proofs for  $N>1$}
As we have said before, all the proofs written here are also valid when $N$ is greater than one, but in order to convince the reader we will at least formulate the numerical scheme in this case.

 We need to introduce some multi index  notation for the spatial discretization, $i=(i_1,. . . , i_N)$  with $i_l=0, . . ., I_l$ for  $l=1, . . . , N$ where $I_l$ is the number of nodes of our mesh in the $l$-esim dimensional direction. We will also use say that
\[1_l=(0,. . . , \underbrace{1}_{l-esim}, . . . ,0)\]
In this way, Problem (\ref{numericmethod}) becomes
 
 \begin{equation}\label{numericmethodRN}\displaystyle
\left\{ \begin{array}{ll}\displaystyle
\frac{\displaystyle \sum_{l=1}^N \big[(W_j)_{i+1_l}^{k}+(W_j)_{i-1_l}^{k}\big]+(W_j)_{i}^{k+1}+(W_j)_{i}^{k-1}-2(N+1)(W_j)_{i}^{k}}{\Delta x^2} =0, &  0<i<I,\\ & 0<k<K,\\[3mm]\displaystyle
(W_j)_i^0=\bigg[\frac{\Delta t}{\Delta x} \big({(W_{j-1})}_i^1-{(W_{j-1})}_i^0\big)  +[(W_{j-1})_i^0]^{1/m}\bigg]^m, & 0<i<I,\\[3mm]\displaystyle
(W_j)_i^k=0,& \mbox{otherwise, }\\
\end{array} \right.
\end{equation}
where the solution of

\begin{equation}\nonumber\displaystyle
\left\{ \begin{array}{ll}\displaystyle
\frac{ \sum_{l=1}^N \big[(W_j)_{i+1_l}^{k}+(W_j)_{i-1_l}^{k}\big]+(W_0)_{i}^{k+1}+(W_0)_{i}^{k-1}-2(N+1)(W_0)_{i}^{k}}{\Delta x^2} =0, &  0<i<I,\\
& 0<k<K,\\[3mm]\displaystyle
(W_0)_i^0=f^m(x_i), & 0<i<I, \\[3mm]\displaystyle
(W_j)_i^k=0,& \mbox{otherwise, }\\
\end{array} \right.
\end{equation}
is used to start the numerical method.

With this multi-index notation the proofs for the local truncation error, existence, uniqueness  and convergence are valid without any change.

\subsection{Signed initial data}

We have assume that the initial data $f$ is a nonnegative function only for simplicity. One can easy observe that all the proofs are valid for $f$ with sign with very small changes. Again, problems could come proving the required regularity  for the theoretical solutions. Lets call 
\[ b_{min}= \min_{x}\{f^m(x),0\}.\] 
Then, in \textbf{Theorem \ref{maxprin}} the same argument holds to prove $(U_j)_i^0\geq b_{min}$ and so, the new maximum/minimum principle states, under the same assumptions of the previous one, that for all $i,j,k$
\[b_{min}\leq (U_j)_i^0\leq b_{max}. \]
In \textbf{Theorem \ref{convergence}} nothing change since we are always talking about errors, and so we are working with absolute values. 

\section{Avoiding the regularity problems: Comparison with the problem in the whole $\mathbb{R}^{N+1}_+$}
As we have said before, in most of the proofs we have assumed high regularity for the theoretical solution in the bounded domain. This fact allow us to use the required order taylor expansion. This kind of regularity results are proven in \cite{afracpor} for the problem posed in $\mathbb{R}^{N+1}_+$. In this section we propose another numerical approach that avoids the regularity problem in the bounded domain but in return gives an extra condition for the convergence.

We will compare the solution to the numerical scheme (\ref{numericmethod}) posed in the bounded domain $[-X,X]\times[0,X]$  with the theoretical solution to (\ref{problemext}) posed in $\mathbb{R}^{N+1}_+$.  Obviously, the comparison is only done where the numerical scheme is well defined.

A new difficulty appears with this comparison. Now $w\not = 0$ in $\Gamma_h$ and so we cannot have convergence for a fixed domain. The idea is making $\Omega \lra \mathbb{R}^{N+1}_+$  as $\Delta x\lra 0$ with a certain rage of convergence.
Note that an extra error coming from the lateral boundary will be introduced.

In \cite{jlbar}, an upper bound for the solution with compactly supported initial data is found passing through the Barenblatt solutions of problem (\ref{problem}). The upped bound is,

\[u_M^*(x,t)=(t+1)^{-\alpha} F(|x|(t+1)^{-\beta}),\]
where $F(\xi)\leq C |\xi|^{N+1}$ and
\[\alpha=\frac{N}{N(m+1)+1}, \ \ \ \ \beta=\frac{1}{N(m+1)+1}.\]
Since $-\alpha+\beta(N+1)=\beta$, we have the next bound in $\Gamma_h$ depending only on $X$,
\[u_M^*(X,t)\leq C(t+1)^{-\alpha+\beta(N+1)}\frac{1}{X^{N+1}}\leq C(t+1)^{\beta}\frac{1}{X^{N+1}}\leq C\cdot (T+1)^\beta\frac{1}{X^2}.\]
Then, if we impose the following condition in the domain, 
\begin{equation}
|X|\geq K\frac{1}{\Delta x},
\end{equation}
we can adapt the proofs of Theorems \ref{localtruncerror} and \ref{convergence} to obtain the whised convergence.

In \textbf{Theorem \ref{localtruncerror}}, the local truncation error in the interior of $\Omega$ nodes and in $\Gamma_d$ still being the same but is not zero anymore in $\Gamma_h$. Now if $(x_i,y_k) \in \Gamma_h$,
\[(\tau_j)_i^k=(w_j)_i^k\leq C \frac{1}{|X|^2}\leq D\Delta x^2.\]
and so $\Lambda= O(\Delta x^2)$ as before.

In \textbf{Theorem \ref{convergence}}, again the unique change is that the error in $\Gamma_h$ is not zero. But, if $(x_i,y_k) \in \Gamma_h$,
\[(e_j)_i^k=(w_j)_i^k-(W_j)_i^k=(w_j)_i^k\leq D\Delta x^2.\]
And so $E_J=O(\Delta x+ \Delta t)$.

We get then the next result,
\begin{teor}
Let $w$ be the solution to Problem (\ref{problemext}) (posed in $\RR$) and $(W_j)_i^k$  be the solution to system (\ref{numericmethod}) (posed in the bounded domain $\Omega=[-X,X]\times[0,X]$) with $m\geq 1$ and compactly supported initial data $f\in L^1(\RR)$. Assume that:

1. There exists two constants $C(m,f), D>0$ such that 
\[D\Delta x\leq \Delta t\leq C(m,f)\Delta x. \]

2. There exists a constant $K>0$ such that
\[|X|\geq K\frac{1}{\Delta x}\]
Then
\[\max_{i,j,k}|w(x_i,y_k,t_j)-(W_j)_i^k|=O(\Delta x+\Delta t).\]
\end{teor}
\begin{nota}
Condition 2 says that $|X|\stackrel{\Delta x \to 0}{\lra} \infty$ and so $\Omega \stackrel{\Delta x \to 0}{\lra} \RR$.
\end{nota}
\section{Extension to a more general fractional diffusion equation}
It is also possible to formulate a numerical method for more general equations 
\begin{equation}\label{problemphi}
\left\{
\begin{array}{ll}
\displaystyle\frac{\partial u}{\partial t}(x,t)+\laph \varphi(u)(x,t)=0,& x\in \RR, \ t>0,\\[3mm]
u(x,0)=f(x), & x \in \RR,
\end{array}
\right.
\end{equation}
if $\varphi\in C^2(\rr)$ such that $\varphi'\geq 0$ and $\varphi', \varphi''$ locally bounded. We also need existence of  $\varphi^{-1}\in C^2(\rr)$. In this case, we have an associated extension problem
\begin{equation}\label{problemextphi}\displaystyle
\left\{ \begin{array}{ll}\displaystyle
\Delta w(x,y,t)=0,& x \in \RR,\ y>0,\ t>0, \\[3mm]
\displaystyle\frac{\partial \varphi^{-1}(w)}{\partial t}(x,0,t)=\frac{\partial w}{\partial y}(x,0,t),& x \in \RR,\ y=0,\ t>0, \\[3mm]
w(x,0,0)=\varphi(f(x)) , & x\in \RR,
\end{array} \right.
\end{equation}
as before. In \cite{afracpor3} they study the theory for this problem with 
\[\varphi(u)=\log(u+1),\]
as the natural limit as $m\to 0$ of the problem with \[\varphi(u)=\frac{(u+1)^m-1}{m}.\]
We can also state the corresponding finite difference method associated to this problem posed in the bounded domain,
\begin{equation}\label{numericmethodphi}\displaystyle
\left\{ \begin{array}{ll}\displaystyle
\frac{(W_j)_{i+1}^{k}+(W_j)_{i-1}^{k}+(W_j)_{i}^{k+1}+(W_j)_{i}^{k-1}-4(W_j)_{i}^{k}}{\Delta x^2} =0, &  0<i<I,
 0<k<K\\[3mm]\displaystyle
(W_j)_i^0=\varphi\bigg(\frac{\Delta t}{\Delta x} \big({(W_{j-1})}_i^1-{(W_{j-1})}_i^0\big)  +\varphi^{-1}[(W_{j-1})_i^0]\bigg), & \mbox{if } 0<i<I\\[3mm]\displaystyle
(W_j)_i^k=0,& \mbox{otherwise, }\\
\end{array} \right.
\end{equation}
where the solution of

\begin{equation}\nonumber\displaystyle
\left\{ \begin{array}{ll}\displaystyle
\frac{(W_0)_{i+1}^{k}+(W_0)_{i-1}^{k}+(W_0)_{i}^{k+1}+(W_0)_{i}^{k-1}-4(W_0)_{i}^{k}}{\Delta x^2} =0, &  0<i<I,\
 0<k<K\\[3mm]\displaystyle
(W_0)_i^0=\varphi(f(x_i)), & \mbox{if } 0<i<I,\\[3mm]\displaystyle
(W_0)_i^k=0,& \mbox{otherwise, }\\
\end{array} \right.
\end{equation}
is used to start the numerical method. 

All proofs can be adapted, without any extra effort, to this method. Of course we have to formulate a more general constant (\ref{constant}) related to $\varphi$. We recall that the restriction for this constant comes from the required positivity of the coefficient $1-\varphi'(\xi)\Delta t/\Delta x$ that appears in (\ref{recu2}). This constant becomes
\begin{equation}\label{constantphi1}
C(\varphi,f)=\bigg(\max_{x\in[0,b_{max}]}\{\varphi'(x)\}\bigg)^{-1}.
\end{equation}
\section{Numerical results}
\subsection{Analysis of the errors}
We present now two different error analysis of the numerical solutions. The first one is obtained comparing solutions with some decreasing $\Delta x$ with a solution generated with a very small $\Delta x$. The results are presented in Figure \ref{err1} and Table 1.
\begin{figure}[h!]
	\begin{center}
		\includegraphics[width=0.9\textwidth]{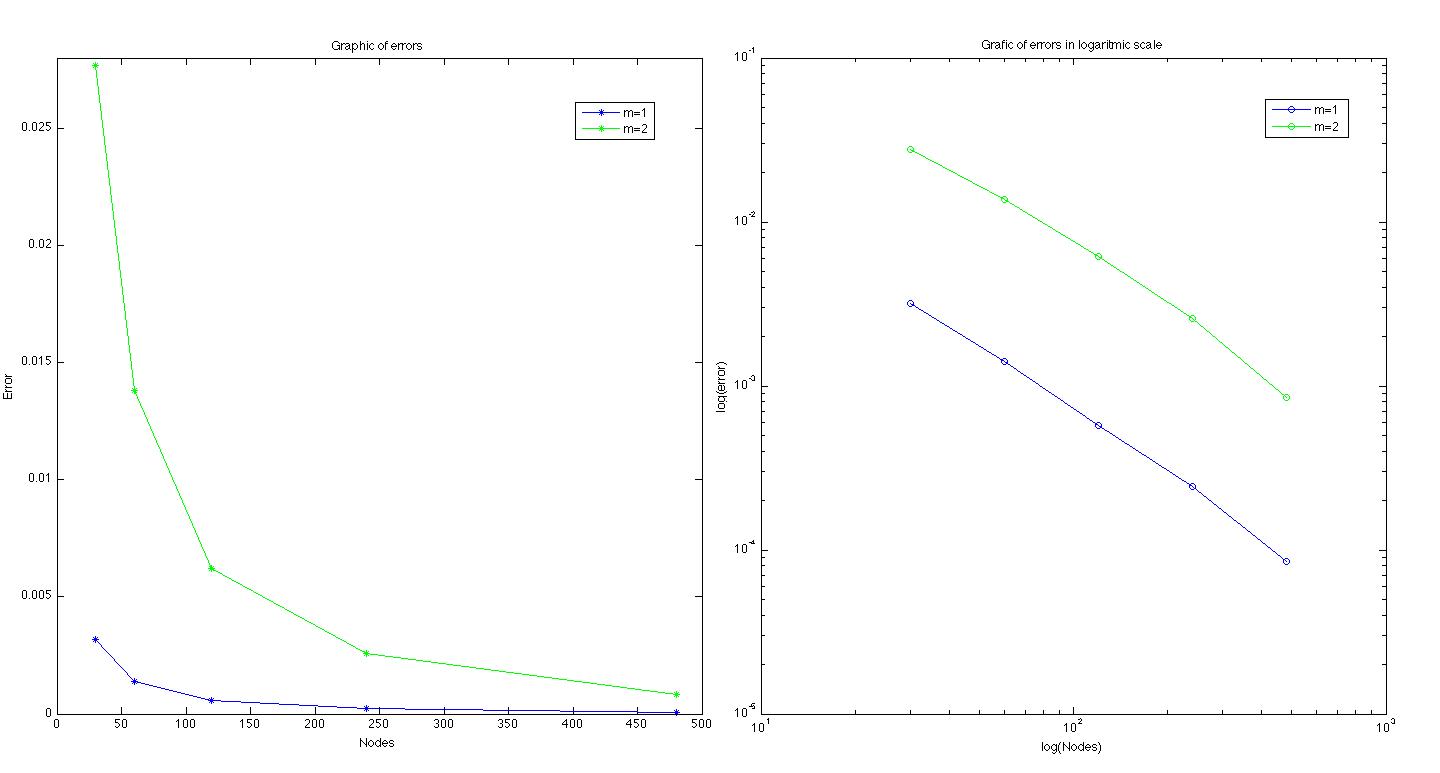}
		\caption{First analysis of the error for m=1 and m=2}     
        		\label{err1}
	\end{center}
\end{figure}

\begin{table}[h!]
$$
 \begin{tabular}{|c||c|c|c|}
  \hline
 $m$ & $\Delta x=\Delta t $&$Nodes$ & $Error$\\
 \hline
 \hline
  1	&  0.2 	&30     & 0.0031733       \\
   	&   0.1 	& 60           &0.0014071                \\
   	&  0.05	&120    &    0.0005760            \\
     	&  0.025	& 240     &  0.0002443             \\
   	&  0.0125	& 480     &  0.0000852      \\
  \hline
 \end{tabular}
 $$
 \end{table}
 \begin{table}[h!]
 $$
 \begin{tabular}{|c||c|c|c|}
  \hline
 $m$ & $\Delta x=\Delta t $&$Nodes$ & $Error$\\
 \hline
 \hline
   2	&  0.2 	&30     & 0.0276966      \\
   	&   0.1 	& 60           &0.0137975                \\
   	&  0.05	&120    &    0.0061930            \\
     	&  0.025	& 240     &  0.0025897             \\
   	&  0.0125	& 480     &  0.0008505      \\
\hline
 \end{tabular}
$$
\caption{First analysis of the error for $m=1$ and $m=2$.}    
\end{table}

The second way of computing the errors can only be done for the case $m=1$ since we have and explicit solution when the initial data is a Dirac delta. The idea is to take as initial data for our method the explicit solution given by
\[u(x,t)=C_N\frac{t}{|x|^2+t^2},\]
at time t=1. So the error will be be computed with the difference of the solution of the method with initial data 
\[f(x)=C_N\frac{1}{|x|^2+1}\]
at time $T=1$ and the real solution at time $T=1$, that is, the solution of the problem with initial data the dirac Delta at time $T=2$. We are going to compare the a real solution in the whole space with a numerical solution computed in a bounded domain so we choose a very large domain in order to minimize the errors that comes from the tails. The chosen domain is $\Omega=[-100,100]\times[0,100]$.

\begin{figure}[h!]
	\begin{center}
		\includegraphics[width=0.9\textwidth]{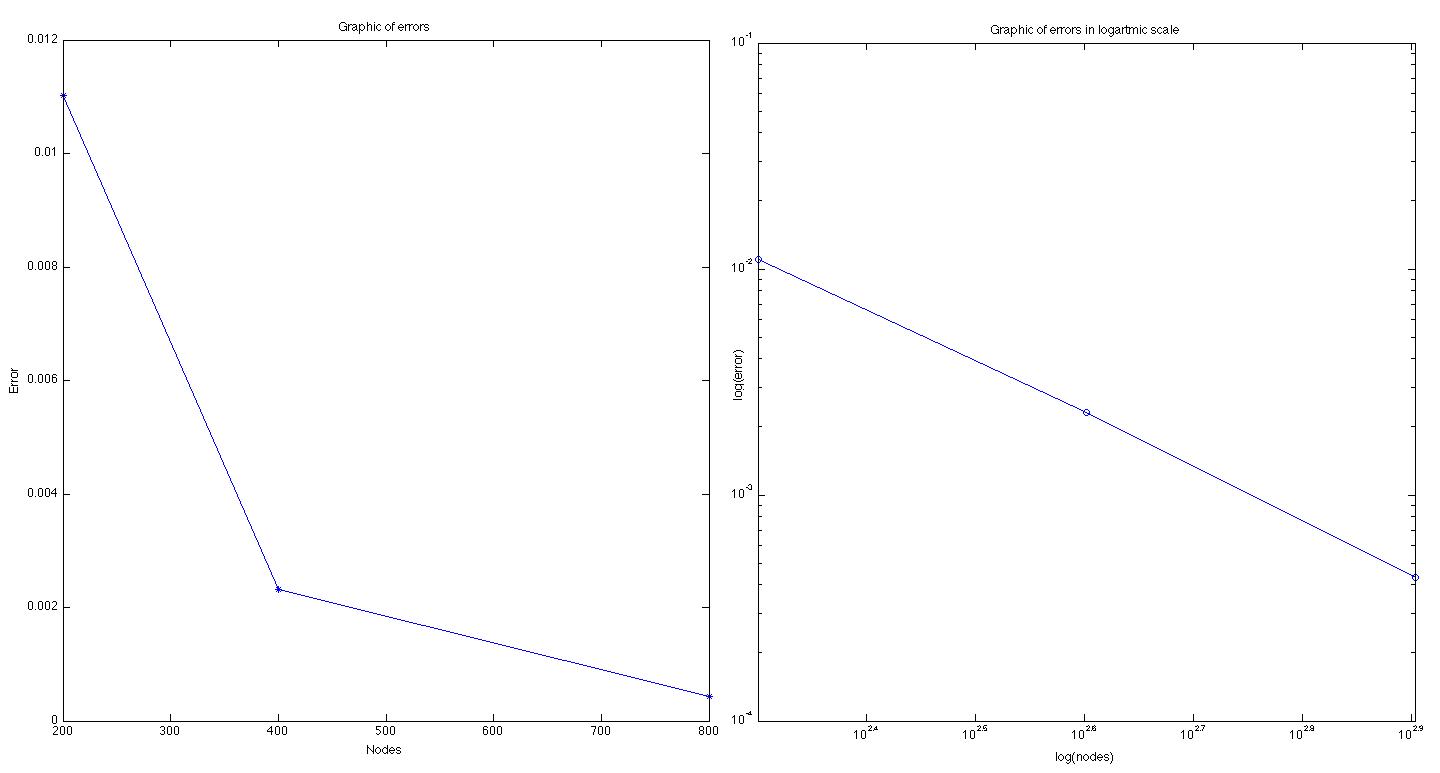}
		\caption{Second error analysis of the error for m=1.}     
        		\label{m1}
	\end{center}
\end{figure}
\begin{table}[h!]
$$
 \begin{tabular}{|c||c|c|c|}
  \hline
 $m$ & $\Delta x=\Delta t $&$Nodes$ & $Error$\\
 \hline
 \hline
  1	&  1 	& 200     & 0.0110235\\
   	&   0.5 	& 400           &0.0023161                \\
   	&  0.25	& 800    &    0.0004329           \\
\hline
 \end{tabular}
$$
\caption{Second analysis of the error for $m=1$}    
\end{table}

A third way of computing errors for $m\not=1$ is possible thanks to the Barenblatt formula introduced by J.L. V\'azquez in \cite{jlbar}. In that paper this numerical method is used to compute some Barenblatt profiles as next picture shows,
\begin{figure}[h!]
	\begin{center}
		\includegraphics[width=.8\textwidth]{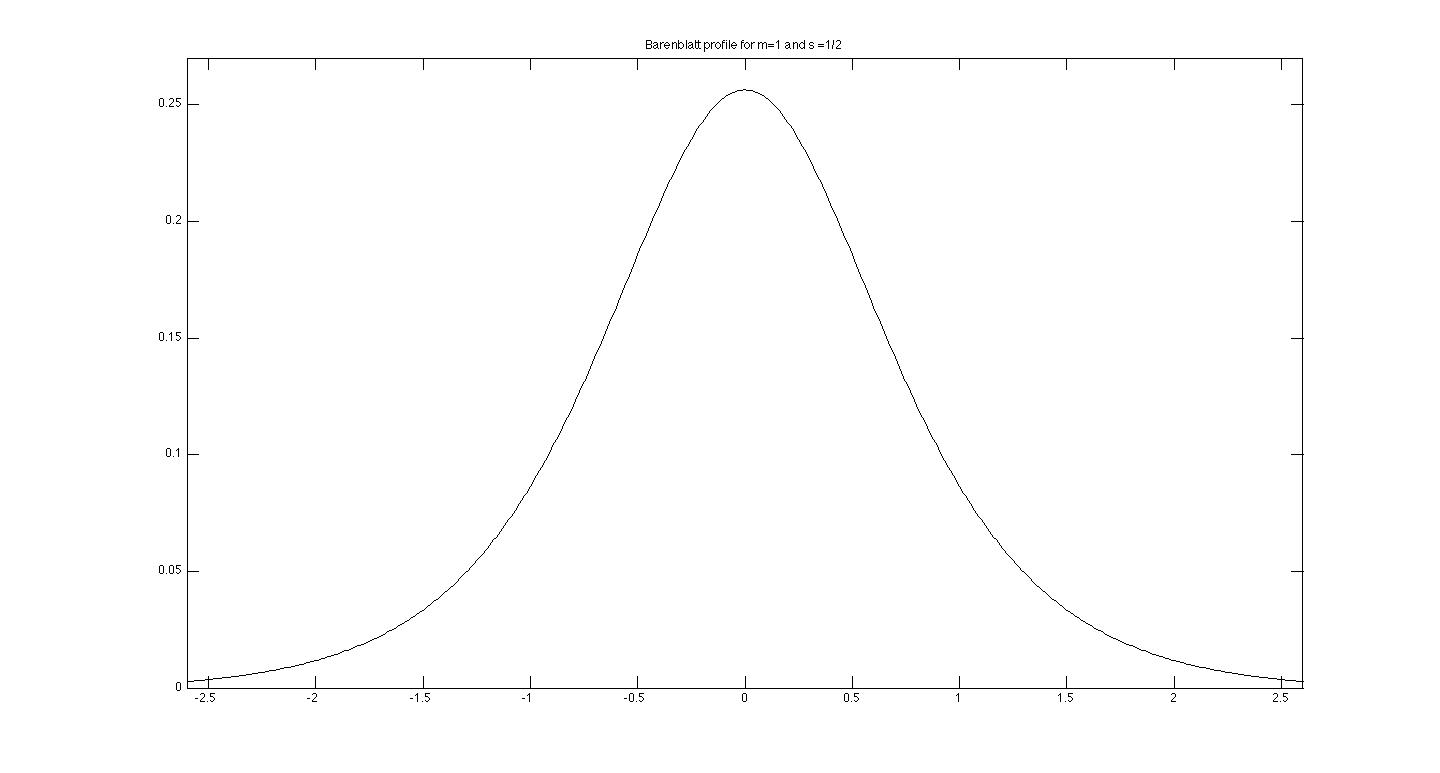}
        		\label{m1}
	\end{center}
\end{figure}
\begin{figure}[h!]
	\begin{center}
		\includegraphics[width=.8\textwidth]{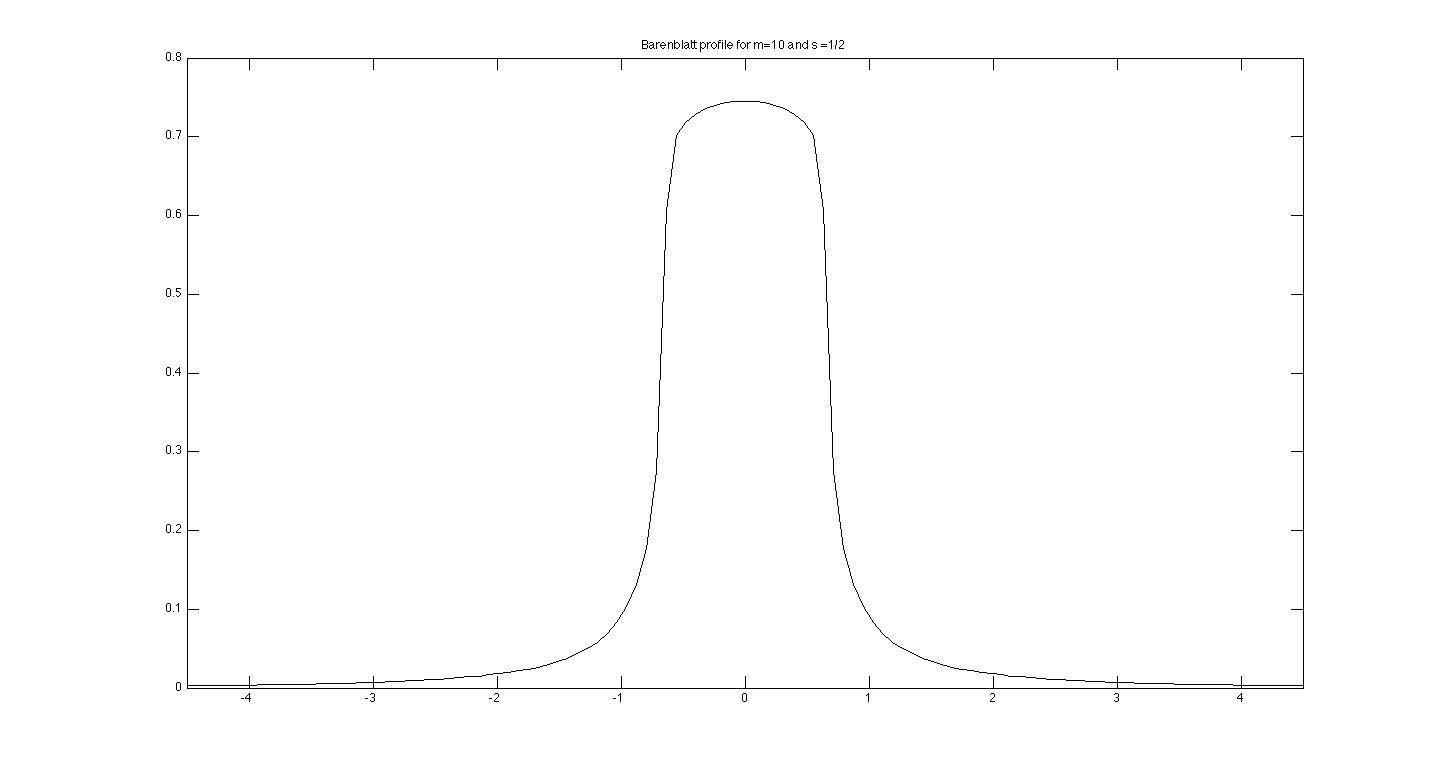}
        		\label{m1}
		\caption{Computed Barenblatt profiles for $m = 1$ and $m=10$ with $s = 1/2$}     
	\end{center}
\end{figure}

\newpage
\subsection{Graphics of some solutions}
We next present some graphics with the numerical results obtained with initial data
\[f(x)=C e^{-\frac{1}{(1-x)(1+x)}} \chi_{[-1,1]}(x).\]
In Figures 4 and 5 we show the numerical solutions for two small $m$ where the expected fat tail is observed.
\begin{figure}[h!]
	\begin{center}
		\includegraphics[width=.8\textwidth]{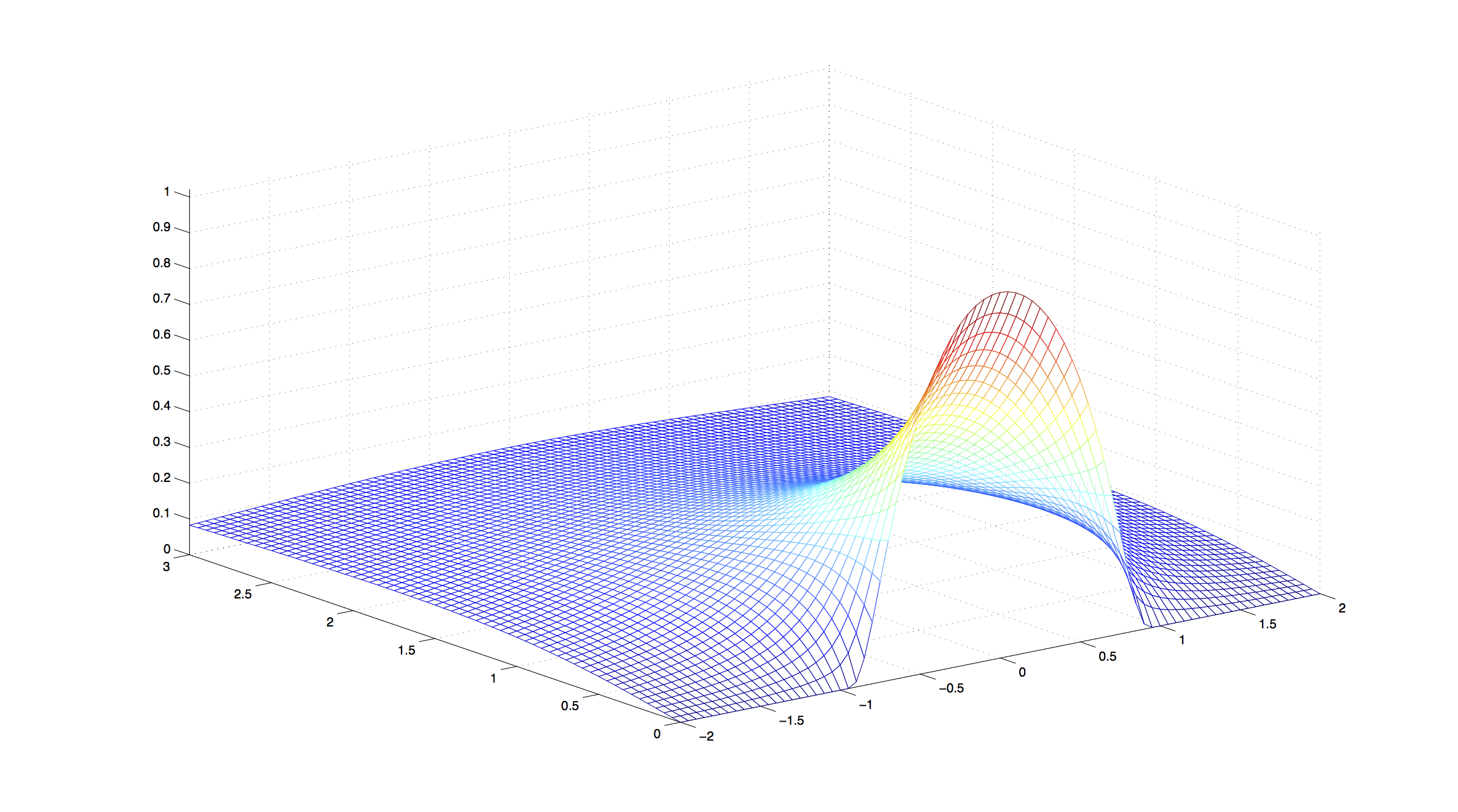}
		\caption{Numerical solution for $m=1$}     
        		\label{m1}
	\end{center}
\end{figure}

\begin{figure}[h!]
	\begin{center}
		\includegraphics[width=0.8\textwidth]{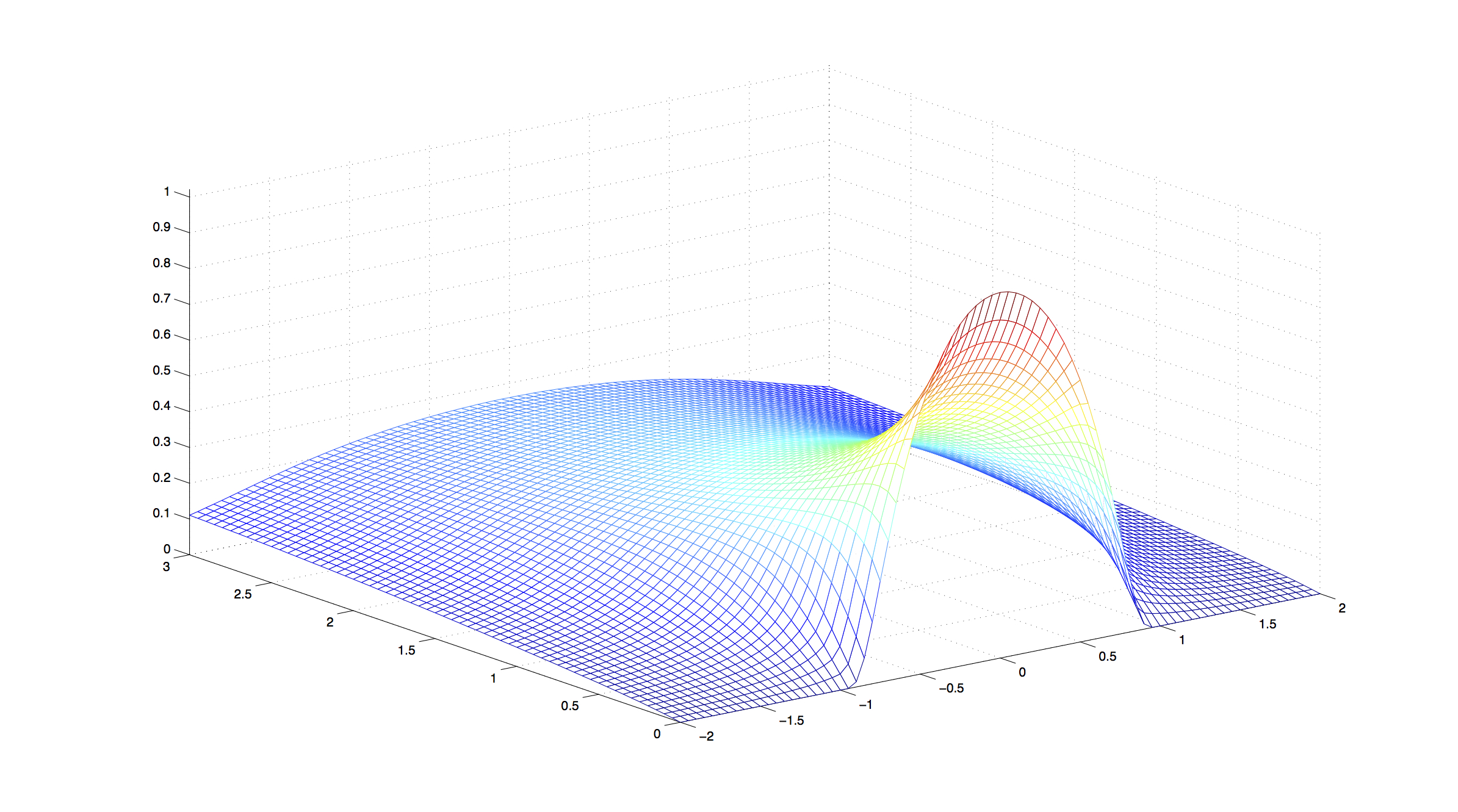}
		\caption{Numerical solution for $m=2$}     
        		\label{m2}
	\end{center}
\end{figure}
In the Figure 6 we observe the typical very slow diffusion of the porous medium equation with high values of $m$.
\begin{figure}[h!]\label{fig3}
	\begin{center}
		\includegraphics[width=0.95\textwidth]{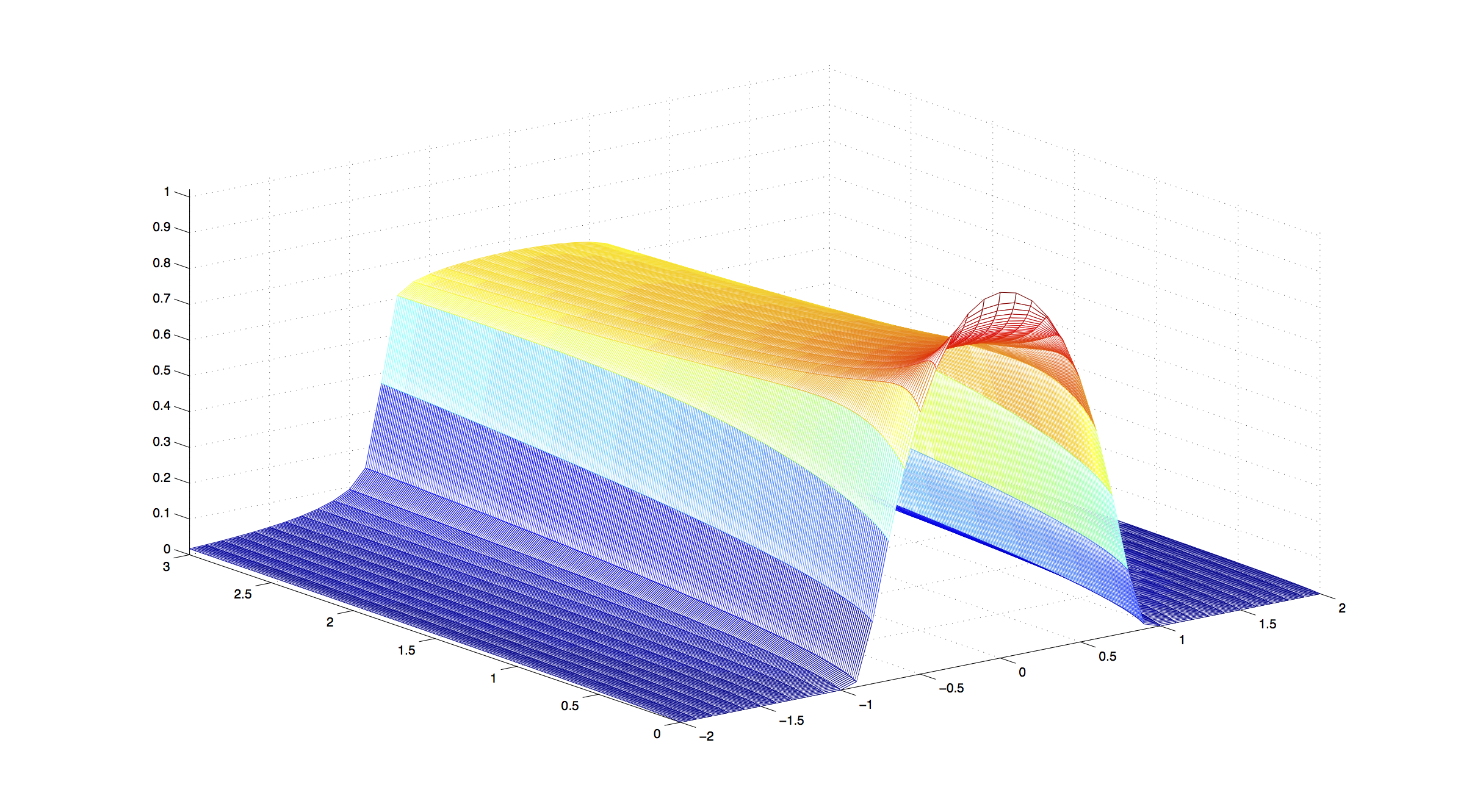}
		\caption{Numerical solution for $m=10$}     
        		\label{m10}
	\end{center}
\end{figure}

A case with different diffusion is presented in Figure 7 as an example where the numerical  method is used for more general $\varphi$.
\begin{figure}[h!]
	\begin{center}
		\includegraphics[width=0.85\textwidth]{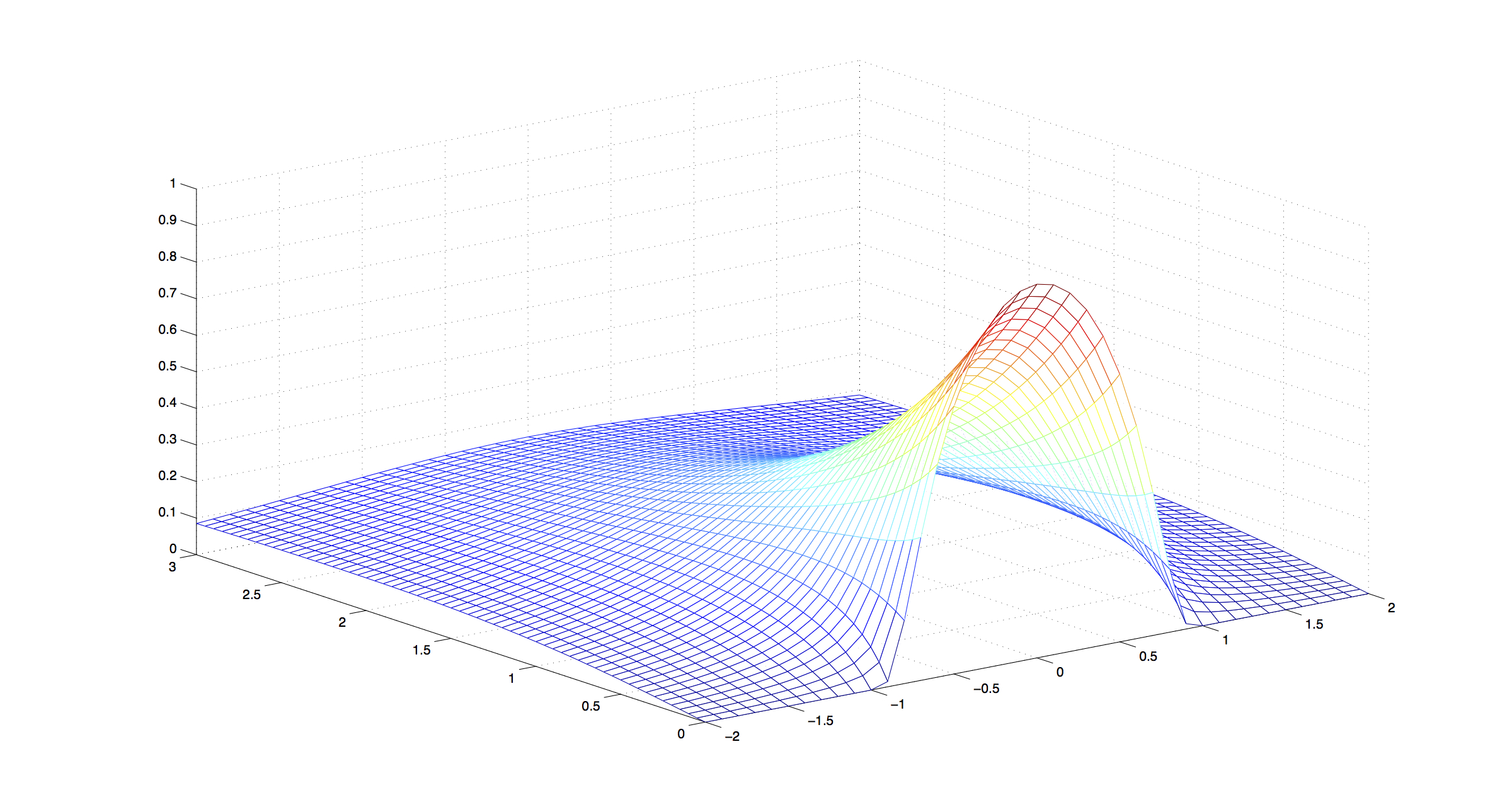}
		\caption{Numerical solution for $\varphi(u)=log(u+1)$}     
        		\label{m10}
	\end{center}
\end{figure}
\section*{Acknowledgments}
The author partially supported by the  Spanish Project MTM2011-24696 and by a FPU grant from Ministerio de Educaci\'on, Ciencia y Deporte, Spain.


\medskip



\begin{thebibliography}{99} 
\bibitem{caff}L. A. Caffarelli and L. Silvestre. \em An extension problem related to the fractional
laplacian\em, Comm. Partial Differential Equations 32 (2007), 12451260.
\bibitem{jakob}   S. Cifani, E. R. Jakobsen, and K. H. Karlsen.
\em The discontinuous Galerkin method for fractional degenerate convection-diffusion equations.\em
BIT 51(4), 809-844, 2011.
\bibitem{jakob2} S. Cifani and E. R. Jakobsen.
\em On the spectral vanishing viscosity method for periodic fractional conservation laws. \em To appear in Math. Comp.
\bibitem{jakob3}S. Cifani, and E. R. Jakobsen.
On numerical methods and error estimates for degenerate fractional convection-diffusion equations.
Submitted 2012.
\bibitem{landkof} N. S. Landkof. \em Foundations of modern potential theory\em. Springer-Verlag, New York,
1972. Translated from the Russian by A. P. Doohovskoy, Die Grundlehren der mathematischen
Wissenschaften, Band 180.
\bibitem{afracpor} A. De Pablo, F. Quir\'os, A. Rodr\'iguez, J. L V\'azquez.  \em A fractional porous medium equation.\em
Adv. Math. 226 (2011), no. 2, 1378Ð1409.
\bibitem{afracpor2}A. De Pablo, F. Quir\'os, A. Rodr\'iguez, J. L V\'azquez. \em A general fractional porous
medium equation\em. To appear in Comm. Pure Appl. Math., arXiv:1104.0306v1.
\bibitem{afracpor3}A. De Pablo, F. Quir\'os, A. Rodr\'iguez, J. L V\'azquez. \em Classical solutions for a logarithmic fractional diffusion equation\em. http://arxiv.org/pdf/1205.2223.pdf.
\bibitem{PQRV4}A. De Pablo, F. Quir\'os, A. Rodr\'iguez, J. L V\'azquez. In preparation.
\bibitem{yotov} C. Hall,  T. Porsching, \em Numerical Analysis of Partial Differential Equations\em  Prentice Hall, 1990. 
\bibitem{val}E. Valdinoci. \em From the long junp random walk to the fractional laplacian.\em Bol. Soc. Esp. Mat. Apl. S\'eMA No. 49 (2009), 33-44.
\bibitem{NOS13}R. H. Nochetto, E. Otarola, A. J. Salgado.
\em A PDE approach to fractional diffusion in general domains: a priori error analysis.\em
 http://arxiv.org/pdf/1302.0698.pdf
\bibitem{jlbar} J.L. V\'azquez. \em Barenblatt solutions and asymptotic behaviour for a nonlinear fractional heat equation of porous medium type\em. http://arxiv.org/pdf/1205.6332v2.pdf

\end{thebibliography}
\end{document}